\documentclass[reqno]{amsart}
\usepackage{amssymb,amsmath,amsfonts,amsthm,commath}
\usepackage{graphicx}
\usepackage{float}
\usepackage{framed}
\usepackage{calc,extarrows}
\usepackage{xcolor}

\usepackage{amsfonts}
\usepackage{amssymb}
\usepackage{graphicx}
\usepackage{newlfont}
\usepackage{mathrsfs}
\usepackage{graphics}
\usepackage{amsmath, amssymb, amsfonts,amsthm}
\usepackage{textcomp}
\usepackage{color}
\usepackage{mathtools}
\usepackage{multicol}
\usepackage{bbm}
\usepackage{float}
\usepackage{enumerate}

\usepackage{hyperref}
\definecolor{bluepigment}{rgb}{0.2, 0.2, 0.6}
\hypersetup{
	colorlinks=true,
	linkcolor=bluepigment,
	citecolor=bluepigment,
	filecolor=black,      
	urlcolor=black,
}

\usepackage[numbers,sort&compress]{natbib}
\usepackage{marginnote}

\newtheorem{theorem}{Theorem}[section]

\theoremstyle{remark}

\numberwithin{equation}{section}

\newtheorem{lemma}[theorem]{Lemma}

\newtheorem{remark}{Remark}[section]

\newcommand{\N}{\mathbb N}
\newcommand{\R}{\mathbb R}

\renewcommand{\P}{\mathbb P}

\newcommand{\E}{\mathbb E}

\newcommand{\prob}{\xrightarrow{\,p\,}}
\newcommand{\distconv}{\xrightarrow{\,d\,}}
\newcommand{\disteq}{\overset{d}{=}}

\newcommand{\Ni}{{\cal N}}

\newcommand{\Yi}{{\cal Y}}

\begin{document}

	\title[Last Progeny Modified Time Inhomogeneous BRW]{Right-Most Position of a Last Progeny Modified Time Inhomogeneous Branching Random Walk}
	
	\author[Bandyopadhyay]{Antar Bandyopadhyay} 
	\address[Antar Bandyopadhyay]{Theoretical Statistics and Mathematics Unit \\
		Indian Statistical Institute, Delhi Centre \\ 
		7 S. J. S. Sansanwal Marg \\
		New Delhi 110016 \\
		INDIA}
%	\address{Theoretical Statistics and Mathematics Unit, 
%		Indian Statistical Institute, Kolkata;
%		203 B. T. Road, Kolkata 700108, INDIA}
	\email{antar@isid.ac.in}          
	\author[Ghosh]{Partha Pratim Ghosh}  
	\address[Partha Pratim Ghosh]{Theoretical Statistics and Mathematics Unit \\
		Indian Statistical Institute, Delhi Centre \\ 
		7 S. J. S. Sansanwal Marg \\
		New Delhi 110016 \\
		INDIA}         
	\email{p.pratim.10.93@gmail.com} 
	
	\date{\today}
	
		\begin{abstract}
		In this work, we consider a modification of time \emph{inhomogeneous} branching random walk, where the driving increment distribution  changes
		  over time macroscopically. 
          Following Bandyopadhyay and Ghosh \cite{BaGh21},		  
		  we give certain independent and identically distributed  (i.i.d.) displacements to all the particles at the last generation. We call this process \emph{last progeny modified time inhomogeneous branching random walk (LPMTI-BRW)}. Under very minimal assumptions on the underlying point processes of the displacements, we show that the maximum displacement converges to a limit after only an appropriate centering which
		  is either linear or linear with a logarithmic correction. 		  
		  Interestingly, the limiting distribution depends only on the first set of increments. We also derive Brunet-Derrida -type results of point process convergence of our LPMTI-BRW to a decorated Poisson point process. As in the case of the maximum, the limiting point process also depends only on the first set of increments. Our proofs are based on the method of coupling the maximum displacement with the smoothing transformation, which was introduced by Bandyopadhyay and Ghosh~\cite{BaGh21}.
	\end{abstract}
	
	\keywords{Branching random walk; time inhomogeneous environments; maximum operator; smoothing transformation; Decorated Poisson point process} 
	
	\subjclass[2010]{Primary: \#60F05; Secondary: \#60G50}
	
	\maketitle
	
	%\tableofcontents
	
	\setlength{\parindent}{0em}
	\setlength{\parskip}{1em}
	
\section{Introduction}
\label{Sec:Intro}
One dimensional \emph{Branching random walk (BRW)} was introduced by Hammersley \cite{Hamm74} in the early '70s and since then it has received 
significant attention from various researches. Some references on this classical
model with \emph{homogenious} displacements which are relevant to our work are \cite{King75, Bigg76, Bram78, BiKy97, BramZei07, HS09, BramZei09, AddRee09, Aide13, AiSh14, Mada2017}. Under reasonable assumptions, it is well known from these literature that in the homogeneous case the maximum displacement grows linearly, with a logarithmic correction, and is tight around its median.
The \emph{inhomogeneous} case has received much attention
in recent years \cite{BramZei09, Fa12, FaZei12}. 
Under certain uniform regularuty assumptions, 
Bramson and Zeitouni \cite{BramZei09} and Fang \cite{Fa12} showed that 
in the inhomogeneous case also 
the maximum displacement re-centered around its median is tight. Later
Fang and Zeitouni \cite{FaZei12} showed that in the binary branching with
independent Gaussian displacements the exact coefficients of the centering 
terms, both for the liner term and also for the logarithmic correction term 
differ based on the increasing/decreasing variance of  the time inhomogeneous 
displacements. This particular example is interesting and relevant for our work and 
is described in more details in Section ~\ref{Sec:Example}.

Bandyopadhyay and Ghosh \cite{BaGh21} introduced a new modification of the 
classical homogeneous BRW, where they added a set of i.i.d. displacements 
with a specific form at the last generation. This new process was 
termed as the \emph{last progeny modified BRW (LPM-BRW)} \cite{BaGh21}.
In this work, we consider an inhomogeneous version with the same modification. 
Our results complement and work of Fang and Zeitouni \cite{FaZei12} in the context 
of this new model of last progeny modified version of BRW. We shall show that under mild 
conditions the maximum displacement after appropriate centering, which 
can either be linear or linear with a logarithmic correction, has a weak limit, 
where the limiting distribution depends only on the point process of the first set of 
displacements. This result is unusual and can have some non-trivial statistical 
applications (see Remark ~\ref{Rem:Possible-Stat-Application}).

	\subsection{Model}

	We fix $k\in \N$. For each $i \in\{1,2,\ldots,k\}$, we  let $Z_i$ be a point process with $N_i:=Z_i(\R)<\infty$ a.s. and $q_i$ be a sequence of integers satisfying $\sum_{i=1}^k q_i(n)=n$, and we write $t_m=\sum_{i=1}^m q_i(n)$. A \emph{time inhomogeneous branching random walk (TI-BRW)} is a discrete-time stochastic process that can be described for each $n\geq 1$ as follows:
	
	At the $0$-th generation, we start with an initial particle at the origin. At time $t\in (t_{m-1},t_m] $, each of the particles at generation $(t-1)$ gives birth to a	random number of offspring distributed according to $N_m$. The offspring are then given random displacements independently and according to a copy of	the point process $Z_m$.
	
	For a particle $v$ in the $t$-th generation, we write $|v|=t$ and $S(v)$ denotes its position, which is the sum of all the displacements the	particle $v$ and its ancestors have received. We shall call the process $\{S(v):|v|=t,0\leq t\leq n, n\geq1\}$ a \emph{time inhomogeneous branching random walk (TI-BRW)}. We denote $R_n:=\max_{|v|=n}S(v)$ as the right-most position at the $n$-th generation.
	
Following our earlier work \cite{BaGh21}	, in this model we also 
introduce a non-negative real number $\theta>0$, which should be thought of as a scaling parameter for the additional displacement we give to each particle at the $n$-th generation. The additional displacements are of the form $\frac{1}{\theta}(-\log E_v)$, where $\{E_v\}_{|v|=n}$ are i.i.d. $\mbox{Exponential}\,(1)$ and are independent of the process $\{S(v):|v|\leq n\}$. We denote by $R_n^*\equiv R_n^*(\theta)$ the right-most position of this \emph{last progeny modified time inhomogeneous branching random walk (LPMTI-BRW)}.
	
	\subsection{Assumptions} 
	we first introduce the following important quantities. For each point process $Z_i=\sum_{j\geq1} \delta_{\xi_j^{(i)}}$ with $1\leq i\leq k$, we define
	\begin{equation}
		\nu_i(a):=\log\E\left[\int_{\R}e^{a x}\, Z_i(dx) \right]=\log \E\Bigg[\sum_{i=1}^{N_i} e^{a \xi_j^{(i)} }\Bigg],
	\end{equation}
for $a\in\R$, whenever the expectations exist. Needless to say that for each $i \in\{1,2,\ldots,k\}$, $\nu_i$ is the logarithm of the moment-generating function of the point process $Z_i$. 

Throughout this paper, for each $i \in\{1,2,\ldots,k\}$, we assume the followings:
\begin{itemize}
		
	\item[{\bf (A1)}]
	$\nu_i(a)$ is finite 
	for all  $a \in(-\vartheta,\infty)$ for some $\vartheta>0$.\\
	
	\item[{\bf (A2)}] 
	The point process $Z_i$ is \emph{non-trivial},  
	and the extinction probability of 
	the underlying branching process is $0$, i.e.,  
	$\P(N_i=1)<1$, $\P( Z_i(\{a\})=N_i )<1$ for any $a\in\R$  
	and $\P(N_i\geq1)=1 $.\\
	
	\item[{\bf (A3)}]
	$N_i$ has finite $(1+p)$-th moment for some $p>0$.
\end{itemize}

\subsection{Outline}
In Section ~\ref{Sec:Results}, 
we state our main results, which are proved in Section ~\ref{Sec:Proofs}.
In Section ~\ref{Sec:Example} we consider an important example and compare
our results with that of the existing literature. 

\section{Main Results}
\label{Sec:Results}
We first introduce some constants related to the point processes $Z_i$'s. For $i \in\{1,2,\ldots,k\}$,  we define
\begin{equation}
	\theta_{(i)}:=\inf\Big\{ a>0 : \frac{\nu_i(a)}{a}= \nu_i'(a) \Big\}.
\end{equation}
From our earlier work \cite{BaGh21},  we 
note that $\nu_i$'s are strictly convex under assumption {\bf (A1)} and {\bf (A2)},
thus, the above set is at most singleton. If it is a singleton, then  $\theta_{(i)}$ is the  unique point in $(0,\infty)$ such that a tangent from the origin to the graph of $\nu_i(a)$ touches the graph at $a=\theta_{(i)}$. And if it is empty, then by definition $\theta_{(i)}$ takes value $\infty$, and there does not exist any tangent from the origin to the graph of $\nu_i(a)$ on the right half-plane. 

\subsection{Asymptotic limits}
Our first result is a centered asymptotic limit of the right-most position, which is similar to the results in below-the-boundary case for \emph{last progeny modified BRW (LPM-BRW)} shown by Bandyopadhyay and Ghosh~\cite{BaGh21}. 
\begin{theorem}
	\label{Thm:Asymptotic-1}
	Suppose $q_i(n) \longrightarrow \infty$ for all $1 \leq i \leq k$, then for any 
	$ \theta < \min_i \theta_{(i)} \leq \infty$,  there exists a random variable  
	$H_{\theta,(1)}^{\infty}$
	 depending only on  $\theta$ and $Z_1$,  such that,
	\begin{equation}
		R_n^*(\theta) - \sum_{i=1}^k\frac{q_i(n)\nu_i\left(\theta\right)}{\theta}  
		\distconv H^{\infty}_{\theta,(1)}.
		\label{Equ:Asymptotic-1}
	\end{equation}
\end{theorem}

\begin{theorem}
\label{Thm:Asymptotic-1+1}
Suppose $q_i(n) \longrightarrow \infty$ for all $1 \leq i \leq k$ and 
$ \theta_{(1)}< \min_{i\neq 1} \theta_{(i)} \leq \infty$, then there exists a random variable  
$H_{\theta_{(1)},(1)}^{\infty}$
depending only on   $Z_1$,  such that,
\begin{equation}
	R_n^*\big(\theta_{(1)}\big) - \sum_{i=1}^k\frac{q_i(n)\nu_i\big(\theta_{(1)}\big)}{\theta_{(1)}}  +\frac{1}{2\theta_{(1)}}\log\left(q_1(n)\right)
	\distconv H^{\infty}_{\theta_{(1)},(1)}.
	\label{Equ:Asymptotic-1+1}
\end{equation}
\end{theorem}

\begin{remark}
\label{Rem:Possible-Stat-Application}
It is very interesting to note that the centered asymptotic limit only depends on the 
point process of the first set of displacements. More interestingly, the result
is valid as long as $q_i(n) \longrightarrow \infty$ for all $1 \leq i \leq k$.
In particular, the rate of divergence of 
$q_1(n)$ can be very slow but we will still have the 
centered asymptotic limit depends only on the distribution of $Z_1$.
Thus our model LPMTI-BRW may be used as a 
very efficient 
``\emph{statistical sheave}''
to filter out the distribution of the first set of displacements 
(may be thought as the ``\emph{signal}") 
from a number of others which may 
be considered as ``\emph{noise}" and of much larger in numbers compared to that
of the ``\emph{signal}". 
We thus feel this result may have greater statistical significance. 
\end{remark}

As we will see in the proof of the above theorem (see Section 3), we have a slightly stronger result.  As in Theorem~2.5 of Bandyopadhyay and Ghosh~\cite{BaGh21}, we let 
\[
\hat{H}_{\theta,(1)}^{\infty} = 
\frac{1}{\theta}\log D_{\theta,(1)}^{\infty},
\]
where $D_{\theta,(1)}^{\infty}$ is the unique solution of the 
following \emph{linear recursive distributional equation} with mean $1$.
\begin{equation}
	\Delta\xlongequal{\text{ d }} \sum_{|v|=1}e^{\theta S(v)-\nu_1(\theta)} \Delta_v, 
	\label{Equ:RDE-for-D-theta-infty}
\end{equation}
where $\Delta_v$ are i.i.d. and has the same distribution as that of $\Delta$. 
As in Theorem~2.3 of Bandyopadhyay and Ghosh~\cite{BaGh21}, we also let 
\[
\hat{H}_{\theta_{(1)},(1)}^{\infty} = 
\frac{1}{\theta_{(1)}}\left[\log D_{\theta_{(1)},(1)}^{\infty}+ \frac{1}{2} \log \left( \frac{2}{\pi \sigma_1^2} \right)\right],
\]
where
\begin{equation}
		D_{\theta_{(1)},(1)}^{\infty} \stackrel{\mbox{a.s.}}{=\joinrel=\joinrel=}
	\lim_{n \rightarrow \infty}
	-
	\sum_{\left\vert v \right\vert = q_1(n)} \left(\theta_{(1)} S_v - 
	q_1(n) \nu \big(\theta_{(1)}\big) \right)
	e^{\theta_{(1)} S_v - 
		q_1(n) \nu (\theta_{(1)})},
	\label{Equ:Def-D-theta1-infty}
\end{equation}

\begin{equation}
	\sigma_1^2 := 
	\E\left[\sum_{\left\vert v \right\vert = 1} \left(\theta_{(1)} S_v - 
	 \nu \big(\theta_{(1)}\big) \right)^2
	e^{\theta_{(1)} S_v -  \nu (\theta_{(1)})}
	\right].
	\label{Equ:Def-sigma1}
\end{equation}

Then we have
	\begin{theorem}
	\label{Thm:Asymptotic-2}
	Suppose $q_i(n) \longrightarrow \infty$ for all $1 \leq i \leq k$, then for any 
	$ \theta < \min_i \theta_{(i)} \leq \infty$, 
	\begin{equation}
		R_n^* (\theta) -\sum_{i=1}^k\frac{q_i(n)\nu_i\left(\theta\right)}{\theta}     - \hat{H}^{\infty}_{\theta,(1)} 
		 \distconv
		 -\log E,
		\label{Equ:Asymptotic-2}
	\end{equation}
where $E \sim \mbox{Exponential}\,(1)$.
\end{theorem}

\begin{theorem}
\label{Thm:Asymptotic-2+1}
Suppose $q_i(n) \longrightarrow \infty$ for all $1 \leq i \leq k$ and 
$ \theta_{(1)}< \min_{i\neq 1} \theta_{(i)} \leq \infty$, then
\begin{equation}
	R_n^*\big(\theta_{(1)}\big) - \sum_{i=1}^k\frac{q_i(n)\nu_i\big(\theta_{(1)}\big)}{\theta_{(1)}}  +\frac{1}{2\theta_{(1)}}\log\left(q_1(n)\right) - \hat{H}_{\theta_{(1)},(1)}^{\infty}
	\distconv -\log E,
	\label{Equ:Asymptotic-2+1}
\end{equation}
	where $E \sim \mbox{Exponential}\,(1)$.
\end{theorem}

\begin{remark}
	Note that  
	$H_{\theta, (1)}^{\infty}$ 
	in Theorem~\ref{Thm:Asymptotic-1} has the same distribution as
	$\hat{H}_{\theta,(1)}^{\infty} - \log E$, where 
	$E \sim \mbox{Exponential}\,(1)$ and is independent of
	$\hat{H}_{\theta,(1)}^{\infty}$.
	Similarly, 
	$H_{\theta_{(1)}, (1)}^{\infty}$ 
	in Theorem~\ref{Thm:Asymptotic-1+1} has the same distribution as
	$\hat{H}_{\theta_{(1)},(1)}^{\infty} - \log E$, where 
	$E \sim \mbox{Exponential}\,(1)$ and is independent of
	$\hat{H}_{\theta_{(1)},(1)}^{\infty}$.
\end{remark}

	 As a corollary of the above results, we obtain that if the centering term converges after dividing by $n$, then $R_n^*/n$ has a limit in probability. In particular, we have the following result:
\begin{theorem}
	\label{Thm:Conv_in_probability}
	If for all $1 \leq i \leq k$, $q_i(n) \longrightarrow \infty$ satisfying $\lim_{n \rightarrow \infty} q_i(n)/n=\alpha_i\geq 0$, then
for any $ \theta < \min_i \theta_{(i)} \leq \infty$ and also for $\theta= \theta_{(1)}< \min_{i\neq 1} \theta_{(i)} \leq \infty$,
\begin{equation}
	\frac{R_n^*(\theta)}{n}\prob \sum_{i=1}^k\frac{\alpha_i\nu_i\left(\theta\right)}{\theta}. \label{Equ:Conv_in_probability}
\end{equation}
\end{theorem}

%\begin{remark}
%	Notice that $ \sum_{i=1}^k \alpha_i = 1 $, but it may happen that some of the 
%	$\alpha_i$'s are $0$ if the corresponding $q_i(n)$'s are $o(n)$. In that case, the corresponding point processes will have no influence on the in probability limit of $R_n^*/n$.
%\end{remark}

\subsection{Brunet-Derrida type results}
Here we present results of the type Brunet and Derrida~\cite{BrDerr2011} for our LPMTI-BRW.

For any $ \theta < \min_i \theta_{(i)} \leq \infty$, we define
\begin{equation}
	Z_n(\theta)=\sum_{|v|=n}\delta_{\left\{\theta S(v)-\log E_v-\sum_{i=1}^k q_i(n)\nu_i\left(\theta\right)  -\theta\hat{H}_{\theta,(1)}^{\infty}\right\}}, 
\end{equation}
and for $ \theta_{(1)}< \min_{i\neq 1} \theta_{(i)} \leq \infty$, we define
\begin{equation}
	Z_n\big(\theta_{(1)}\big)=\sum_{|v|=n}\delta_{\left\{\theta_{(1)} S_v-\log E_v-\sum_{i=1}^k q_i(n)\nu_i(\theta_{(1)})+\frac{1}{2}\log \left(q_1(n)\right)-\theta_{(1)}\hat{H}_{\theta_{(1)},(1)}^{\infty}\right\}}, 
\end{equation}
where $\hat{H}_{\theta,(1)}^{\infty}$ and $\hat{H}_{\theta_{(1)},(1)}^{\infty}$ are as  in  Theorems~\ref{Thm:Asymptotic-2} and~\ref{Thm:Asymptotic-2+1}.  
Our first result is the weak convergence of  the point processes $\left(Z_n\left(\theta\right)\right)_{n \geq 0}$, which is similar to the results for LPM-BRW as shown by Bandyopadhyay and Ghosh~\cite{BaGh21}. 
\begin{theorem}
	\label{Thm:Point-Process-Conv}
	Suppose $q_i(n) \longrightarrow \infty$ for all $1 \leq i \leq k$, then
	for any $ \theta < \min_i \theta_{(i)} \leq \infty$ and also for 
	$\theta= \theta_{(1)}< \min_{i\neq 1} \theta_{(i)} \leq \infty$,
	\[Z_n(\theta)\distconv \Yi, \]
	where $\Yi$ is a decorated Poisson point process. In particular,
	$\Yi=\sum_{j\geq1}\delta_{-\log \zeta_j},$
	where $\Ni=\sum_{j\geq1}\delta_{\zeta_j}$ is a homogeneous Poisson point process on $\R_+$ with intensity $1$.
\end{theorem}
The following is a slightly weaker version of the above theorem.
\begin{theorem}
	\label{Thm:Point-Process-Weak-Conv}
	Suppose $q_i(n) \longrightarrow \infty$ for all $1 \leq i \leq k$, then
	for any $ \theta < \min_i \theta_{(i)} \leq \infty$,
	\begin{equation}
	\sum_{|v|=n}\delta_{\left\{\theta S(v)-\log E_v-\sum_{i=1}^k q_i(n)\nu_i\left(\theta\right)  \right\}}\distconv \sum_{j\geq1}\delta_{-\log \zeta_j+\theta\hat{H}_{\theta,(1)}^{\infty} },
\end{equation}
	and for $ \theta_{(1)}< \min_{i\neq 1} \theta_{(i)} \leq \infty$, 
	\begin{equation}
	\sum_{|v|=n}\delta_{\left\{\theta_{(1)} S_v-\log E_v-\sum_{i=1}^k q_i(n)\nu_i(\theta_{(1)})+\frac{1}{2}\log \left(q_1(n)\right)\right\}}
	\distconv \sum_{j\geq1}\delta_{-\log \zeta_j+\theta_{(1)}\hat{H}_{\theta_{(1)},(1)}^{\infty} }, 
	\end{equation}
	where $\Ni=\sum_{j\geq1}\delta_{\zeta_j}$ is a homogeneous Poisson point process on $\R_+$ with intensity $1$, which is independent of the process $\{S(v):|v|\leq n\}$.
\end{theorem}

Let $\Yi_{\max}$ be the right-most position of the point process $\Yi$, and  $\overline{\Yi}$ be the point process $\Yi$ viewed from its right-most position, 
i.e.,
\[ \overline{\Yi} = \sum_{j\geq1}\delta_{-\log \zeta_j-\Yi_{\max}}. \]
Then as a corollary of the above theorem, we get the following result, which confirms the validity of the \emph{Brunet-Derrida Conjecture} for LPMTI-BRW for any $ \theta < \min_i \theta_{(i)} \leq \infty$.
\begin{theorem}
	\label{Thm:BD-Conj}
	Suppose $q_i(n) \longrightarrow \infty$ for all $1 \leq i \leq k$, then
	for any $ \theta < \min_i \theta_{(i)} \leq \infty$ and also for 
	$\theta= \theta_{(1)}< \min_{i\neq 1} \theta_{(i)} \leq \infty$,
	\[\sum_{|v|=n}\delta_{\left\{\theta S(v)-\log E_v-\theta R_n^*(\theta) \right\}}\distconv \overline{\Yi}. \]
\end{theorem}

\section{Proofs of the Main Results}
\label{Sec:Proofs}

\subsection{Proof of Theorems~\ref{Thm:Asymptotic-1},~\ref{Thm:Asymptotic-1+1},~\ref{Thm:Asymptotic-2} and~\ref{Thm:Asymptotic-2+1}}
To prove these theorems, we need the following technical result. We define the linear statistics
\begin{equation}
	W_n(\theta)\equiv W_n(\theta)(q_1(n),\ldots,q_k(n),Z_1,\ldots,Z_k)  :=\sum_{|v|=n}e^{\theta S(v)}.
\end{equation}
%{\color{blue} in the TI-BRW, where the underlying progeny point process is $Z_1$ for the first $q_1(n)$ generations, $Z_2$ for the next $q_2(n)$ generations, etc., and  $Z_k$ for the last $q_k(n)$ generations.
%}

Then we have
\begin{lemma}
	\label{Lem:Inhomo_Wn}
	For any $ \theta < \min_i \theta_{(i)} \leq \infty$ and also for $\theta= \theta_{(1)}< \min_{i\neq 1} \theta_{(i)} \leq \infty$,
	\begin{equation*}
			\frac{W_n(\theta)(q_1(n),\ldots,q_k(n),Z_1,\ldots,Z_k)\cdot e^{-\sum_{i=1}^k q_i(n)\nu_i\left(\theta\right)}}
			{W_{q_1(n)}(\theta)(q_1(n),Z_1)\cdot e^{- q_1(n)\nu_1\left(\theta\right)}}\prob 1
	\end{equation*}
\end{lemma}
\begin{proof}
	Without loss of generality we can assume that $\nu_i(\theta)=0$ for all $i\in\{1,2,\ldots,k\}$. This can be made to satisfy by centering each point process $Z_i$ by $\nu_i(\theta)$. 
	%To simplify the argument, we also write our model as LPMTI-BRW\,$(q_1(n),\ldots,q_k(n),Z_1,\ldots,Z_k)$.

	We prove the lemma by induction. Note that for $k=1$,  the lemma holds trivially.  We assume the lemma holds for $k=m-1$ for some $m\in\N$.
	
	Now, take $k=m$. For each $v$ such that $|v|=q_1(n)$, we define
\begin{equation}
	\overline{W}_{n,v}(\theta) =\sum_{|u|=n, v<u }e^{\theta \left(S(u)-S(v)\right)}.
\end{equation}
Notice that 	$\big\{\overline{W}_{n,v}(\theta)\big\}_{|v|=q_1n}$ are i.i.d. and have the  same distribution as 
$$W_{n-q_1(n)}(\theta)(q_2(n),\ldots,q_m(n),Z_2,\ldots,Z_m),$$
which by our induction hypothesis and Proposition~4.2\,(ii) of Bandyopadhyay and Ghosh~\cite{BaGh21}  converges in probability to $D_{\theta,(2)}^{\infty}$. Since both of them has mean $1$, we also have
\begin{equation}
	W_{n-q_1(n)}(\theta)(q_2(n),\ldots,q_m(n),Z_2,\ldots,Z_m)\xrightarrow{\,L_1\,} D_{\theta,(2)}^{\infty}.
	\label{Equ:L1_conv}
\end{equation}
 Now, observe that
\begin{align}
	&\frac{W_n(\theta)(q_1(n),\ldots,q_m(n),Z_1,\ldots,Z_m)}{W_{q_1(n)}(\theta)(q_1(n),Z_1)} -1\nonumber\\[.25cm]
	=\,& \sum_{|v|=q_1(n)} \frac{e^{\theta S(v)}}{\sum_{|u|=q_1(n)} e^{\theta S(u)}}\left( \overline{W}_{n,v}(\theta)  -1\right).
%	=\,& \sum_{|v|=q_1(n)} \frac{e^{\theta S(v)}}{\sum_{|u|=q_1(n)} e^{\theta S(u)}}\left( \overline{W}_{n,v}(\theta)  -{D_{\theta,(2)}^{\infty}}_v\right)\nonumber\\
%	&\qquad+ \sum_{|v|=q_1(n)} \frac{e^{\theta S(v)}}{\sum_{|u|=q_1(n)} e^{\theta S(u)}}\left( {D_{\theta,(2)}^{\infty}}_v  -1\right),
\end{align}
Now, from~(5.5) and~(5.6) of Bandyopadhyay and Ghosh~\cite{BaGh21}, we know that 
\begin{equation*}
	M_n(\theta):= \max_{|v|=q_1(n)} \frac{e^{\theta S(v)}}{\sum_{|u|=q_1(n)} e^{\theta S(u)}} \prob 0.
\end{equation*}
Let $\mathcal{F}_n$ be the $\sigma$-field generated by $\{S(v):|v|\leq q_1(n)\}$. Then using Lemma~2.1 of Biggins and Kyprianou (1997) \cite{BiKy97}, which is a particular case of Lemma~2.2 in Kurtz (1972) \cite{Kurt72}, we get that for every $0<\varepsilon<1/2$,
\begin{align}
	&\P\left( \left.\left| \frac{W_n(\theta)(q_1(n),\ldots,q_m(n),Z_1,\ldots,Z_m)}{W_{q_1(n)}(\theta)(q_1(n),Z_1)} -1 \right|>\varepsilon\right| \mathcal{F}_n \right)\nonumber\\
	\leq\,\,& \frac{2}{\varepsilon^2}
	\left(
	\int_{0}^{\frac{1}{M_n(\theta)}} M_n(\theta)t\cdot\P\left(  \left|W_{n-q_1(n)}(\theta)(q_2(n),\ldots,q_m(n),Z_2,\ldots,Z_m)- 1\right|>t  \right)\,dt\right.\nonumber\\
	&\qquad+\left. 
	\int_{\frac{1}{M_n(\theta)}}^{\infty} \P\left(  \left|W_{n-q_1(n)}(\theta)(q_2(n),\ldots,q_m(n),Z_2,\ldots,Z_m)- 1 \right|>t \right)\,dt
	\right)\nonumber\\
	 \leq\,\,& \frac{2}{\varepsilon^2}
	 \left(
	 \int_{0}^{\infty} \P\left(  \left|W_{n-q_1(n)}(\theta)(q_2(n),\ldots,q_m(n),Z_2,\ldots,Z_m)- D_{\theta,(2)}^{\infty} \right|>t/2 \right)\,dt
	 \right.\nonumber\\
	 &\qquad+
	 \int_{0}^{\frac{1}{M_n(\theta)}} M_n(\theta)t\cdot\P\left(  \left|D_{\theta,(2)}^{\infty}- 1\right|>t/2  \right)\,dt\nonumber\\
	 &\qquad+\left. 
	 \int_{\frac{1}{M_n(\theta)}}^{\infty} \P\left(  \left|D_{\theta,(2)}^{\infty}- 1 \right|>t/2 \right)\,dt
	 \right) \label{Equ:Kurtz}
\end{align}
By using the dominated convergence theorem, the second and the third term on the right-hand side of~\eqref{Equ:Kurtz} converges to $0$ as $n\rightarrow\infty$, and by~\eqref{Equ:L1_conv}, the first term also tends to $0$ as $n\rightarrow\infty$. Then by taking expectation and using the dominated convergence theorem again, we get
\begin{equation*}
	\lim_{n\rightarrow\infty}\P\left( \left| \frac{W_n(\theta)(q_1(n),\ldots,q_m(n),Z_1,\ldots,Z_m)}{W_{q_1(n)}(\theta)(q_1(n),Z_1)} -1 \right|>\varepsilon\right)=0,
\end{equation*}
which implies
\begin{equation}
	\frac{W_n(\theta)(q_1(n),\ldots,q_m(n),Z_1,\ldots,Z_m)}{W_{q_1(n)}(\theta)(q_1(n),Z_1)}\prob 1.\label{Equ:ratio}
\end{equation}
So if the lemma holds for $k = m-1$, it also holds for $k = m$. Therefore, by using induction we complete the proof. 
\end{proof}	

An argument similar to that of the proof of Theorem~3.6 and also of the proof of Theorems~2.3 and~2.5 yields
\begin{equation}
	\theta R_n^*(\theta)  \disteq \log W_n(\theta) - \log E, \label{Equ:coupling-1}
\end{equation}
and also
\begin{equation}
	\theta R_n^*(\theta) - \log W_n(\theta)  \disteq \log E, \label{Equ:coupling-2}
\end{equation}
where $E\sim\mbox{Exponential}\,(1)$ and is independent of the process $\{S(v):|v|\leq n\}$. Now, Lemma~\ref{Lem:Inhomo_Wn}, equations~\eqref{Equ:coupling-1} and~\eqref{Equ:coupling-2}, together with Proposition~4.2\,(ii) of Bandyopadhyay and Ghosh~\cite{BaGh21} and Theorem~1.1 of A\"{i}d\'{e}kon and Shi \cite{AiSh14}, prove the theorems.

\subsection{Proof of Theorems~\ref{Thm:Point-Process-Conv} and~\ref{Thm:Point-Process-Weak-Conv}}
A similar argument as in the proof of Theorems~2.7 and~2.8 in Bandyopadhyay and Ghosh~\cite{BaGh21}, together with Proposition~4.2\,(ii) of Bandyopadhyay and Ghosh~\cite{BaGh21}, Theorem~1.1 of A\"{i}d\'{e}kon and Shi \cite{AiSh14} and Lemma~\ref{Lem:Inhomo_Wn}, yields Theorems~\ref{Thm:Point-Process-Conv} and~\ref{Thm:Point-Process-Weak-Conv}.

\section{A Specific Example}
\label{Sec:Example}
In this section we consider a
\emph{time inhomogeneous Gaussian displacement binary BRW}, which is a
a specific example of inhomogeneous BRW
introduced by Fang and Zeitouni \cite{FaZei12}. Here we shall consider 
the last progeny modified version of the same example.
To be precise, let 
 $Z_1=\delta_{\xi_{11}}+\delta_{\xi_{12}}$, $Z_2=\delta_{\xi_{21}}+\delta_{\xi_{22}}$, $\xi_{11}$, $\xi_{12}$ are i.i.d. $\mbox{N}\,(0,\sigma_1^2)$, $\xi_{21}$, $\xi_{22}$ are i.i.d. $\mbox{N}\,(0,\sigma_2^2)$ and $q_1(n)=q_2(n)=n/2$. 
 In this case we have
 \[
 \nu_1(t) =\log 2+\frac{\sigma_1^2t^2}{2}  \quad\text{ and }\quad   \nu_2(t) =\log 2+\frac{\sigma_2^2t^2}{2},
 \]
 and
 \[
 \theta_1=\frac{\sqrt{2\log 2}}{\sigma_1}  \quad\text{ and }\quad   \theta_2=\frac{\sqrt{2\log 2}}{\sigma_2}.
 \]
Therefore by the Theorem~\ref{Thm:Asymptotic-1+1}, we obtain that
\begin{theorem}
Assume $\sigma_1 > \sigma_2$, then
the following sequence of random variables
\[
R_n^*\bigg(\frac{\sqrt{2\log 2}}{\sigma_1}\bigg)-n\left(\sigma_1\sqrt{\frac{\log 2}{2}} + \frac{\sqrt{2\log 2}}{4\sigma_1}\left(\sigma_1^2+\sigma_2^2\right) \right)+\log n \left(\frac{\sigma_1}{2\sqrt{2\log 2}}\right)
\]
converges in distribution to a non-trivial distribution which depends only on 
$\sigma_1$. 
\end{theorem}

As comparison we note that in Fang and Zeitouni \cite{FaZei12}, it is shown that
for this example when $\sigma_1 > \sigma_2$, the following sequence of 
random variables
\[
R_n- n\left(\left(\sigma_1+\sigma_2\right)\sqrt{\frac{\log 2}{2}}  \right)
+\log n \left(\frac{3\left(\sigma_1+\sigma_2\right)}{2\sqrt{2\log 2}}\right)
\]
is tight.

Thus for our model we have been able to establish more than Fang and Zeitouni \cite{FaZei12} as we obtain a weak limit for the right-most position of the 
LMPTI-BRW after an appropriate centering. However, we only have this for the 
case when $\sigma_1 > \sigma_2$. In Fang and Zeitouni \cite{FaZei12} the 
other case when $\sigma_1 < \sigma_2$ has also been worked out and
tightness of the right-most position has been proved with an appropriate centering.

\bibliographystyle{plain}	
\bibliography{LPM-TI-BRW}

\end{document}